\documentclass[psamsfonts]{amsart}

\usepackage{amssymb, tikz, url}
\usetikzlibrary{matrix}
\newtheorem{theorem}{Theorem}
\newtheorem{lemma}{Lemma}
\newtheorem{conjecture}{Conjecture}

\newcommand{\del}{\partial}
\newcommand{\abs}[1]{\left\lvert #1 \right\rvert}
\DeclareMathOperator{\const}{const}
\DeclareMathOperator{\Vol}{Vol}
\DeclareMathOperator{\mm}{\mathsf{max-mult}}
\DeclareMathOperator{\straight}{\mathsf{straight}}
\DeclareMathOperator{\symm}{\mathsf{symm}}
\DeclareMathOperator{\sign}{sign}
\DeclareMathOperator{\dist}{dist}

\begin{document}

\title[Simplicial volume and traversing vector fields]{Using simplicial volume to count multi-tangent trajectories of traversing vector fields}
\author{Hannah Alpert}
\author{Gabriel Katz}
\address{MIT\\ Cambridge, MA 02139 USA}
\email{hcalpert@math.mit.edu}
\email{gabkatz@gmail.com}
\begin{abstract}
For a non-vanishing gradient-like vector field on a compact manifold $X^{n+1}$ with boundary, a discrete set of trajectories may be tangent to the boundary with reduced multiplicity $n$, which is the maximum possible.  (Among them are trajectories that are tangent to $\del X$ exactly $n$ times.)  We prove a lower bound on the number of such trajectories in terms of the simplicial volume of $X$ by adapting methods of Gromov, in particular his ``amenable reduction lemma''.  We apply these bounds to vector fields on hyperbolic manifolds.
\end{abstract}
\maketitle

\section{Introduction}

In this paper we consider a smooth vector field $v$ on a space $X$, which is a compact smooth manifold with boundary, with $\dim X = n+1$.  For any such vector field, we may form the space of trajectories, denoted $\mathcal{T}(v)$, of the flow along $v$, and the quotient map is denoted $\Gamma : X \rightarrow \mathcal{T}(v)$.  In general $\mathcal{T}(v)$ may not be a nice space, but it is nicer if $v$ is a \textbf{\emph{traversing vector field}}: a non-vanishing vector field such that every trajectory is either a singleton in $\del X$ or a closed segment.  Figure~\ref{2-dim} depicts a traversing vector field on a $2$-dimensional space, and the associated trajectory space.  One of the authors has explored this general setup in multiple papers beginning with~\cite{Katz09}, and in the paper~\cite{Katz2} he introduces the class of \textbf{\emph{traversally generic}} vector fields, which have certain nice properties.  In Theorem 3.5 of that paper, he proves that the traversally generic vector fields form an open and dense subset of the traversing vector fields.  Therefore, we study only the traversally generic vector fields; the definition and relevant properties appear in our Section~\ref{traversing-setup}.

\begin{figure}\label{2-dim}
\begin{center}
\includegraphics[width=3in]{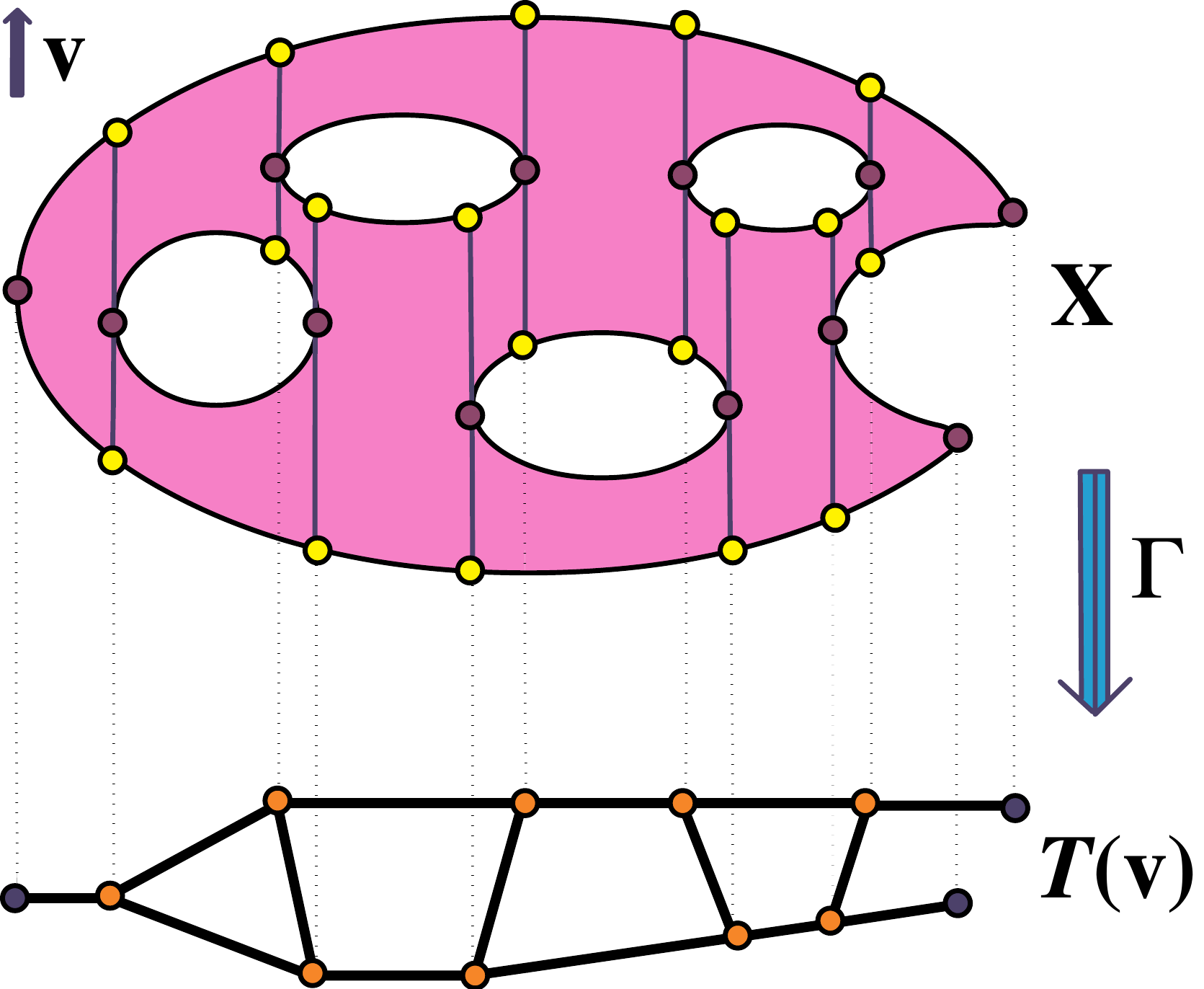}
\end{center}
\caption{The typical trajectory of a traversally generic vector field $v$ is a path meeting $\del X$ twice, each time with multiplicity $1$.  When $X$ is $2$-dimensional, the possible sequences of multiplicities along a trajectory are $(11)$, $(2)$, and $(121)$; in the trajectory space $\mathcal{T}(v)$, these correspond to edge points, vertices of degree $1$, and vertices of degree $3$.}
\end{figure}

Every traversally generic vector field $v$ has a well-defined \textbf{\emph{multiplicity}} $m(a)$ with which $v$ meets $\del X$ at a point $a$, and every trajectory $\gamma$ has a \textbf{\emph{reduced multiplicity}} $m'(\gamma)$, which is the sum over all $a \in \gamma \cap \del X$ of $(m(a) - 1)$.  (The full definition of multiplicity appears in Section~\ref{traversing-setup}.)  Every trajectory of a traversally generic vector field $v$ on a manifold $X^{n+1}$ has reduced multiplicity at most $n$, and so we denote by $\mm(v)$ the set of \textbf{\emph{maximum-multiplicity trajectories}}; that is, those trajectories $\gamma$ with $m'(\gamma) = n$.

\begin{theorem}\label{special-case} 
Let $M$ be a closed, oriented hyperbolic manifold of dimension \mbox{$n + 1 \geq 2$}, and let $X$ be the space obtained by removing from $M$ an open set $U$ satisfying the following properties:
\begin{itemize}
\item The boundary $\del U = \del X$ is a closed submanifold of $M$, possibly with multiple connected components; and
\item The closure $\overline{U}$ is contained in a topological open ball of $M$, possibly very far from round.
\end{itemize}
Let $v$ be a traversally generic vector field on $X$.  Then we have
\[\#\mm(v) \geq \const(n) \cdot \Vol M.\]
\end{theorem}

In particular, because $\Vol M$ is nonzero, there must be at least one maximum-multiplicity trajectory.  This theorem generalizes Theorem~7.5 of~\cite{Katz09}, which addresses the case where $n + 1 = 3$ and $U$ is any finite disjoint union of balls, with constant $1/\Vol(\Delta^3)$, where $\Delta^3$ denotes the regular ideal simplex in hyperbolic $3$-space.

Theorem~\ref{special-case} is a special case of the main theorem of this paper.  The main theorem is a variant of the theorem ``$\Delta$-Inequality for Generic Maps'' in Section 3.3 of Gromov's paper~\cite{Gromov09}.  It requires the notion of simplicial volume, which was introduced in~\cite{Gromov82} and is defined as follows.  For every singular chain $c$ with real coefficients, the norm of $c$, denoted $\Vert c \Vert_\Delta$, is the sum of absolute values of the coefficients.  For every real homology class $h$, the \textbf{\emph{simplicial norm}} (really a semi-norm) of $h$, denoted $\Vert h \Vert_{\Delta}$, is the infimum of $\Vert c \Vert_\Delta$ over all cycles $c$ representing $h$.  The simplicial norm is often called the \textbf{\emph{simplicial volume}} because it generalizes hyperbolic volume: if $M$ is any closed, oriented hyperbolic manifold of dimension $n$, then $\Vol M = \const(n) \cdot \Vert [M] \Vert_\Delta$ (Proportionality Theorem, p.~11 of~\cite{Gromov82}).

Our main theorem is stated as follows.  If $X$ is an oriented manifold with boundary, then let $D(X)$ denote the double of $X$, which is the oriented manifold obtained by gluing two copies of $X$ along their boundary $\del X$.

\begin{theorem}\label{main-theorem}
Let $X$ be a compact, oriented manifold with boundary, with $\dim X = n+1$.  Let $Z$ be a space with contractible universal cover, and let $\alpha: D(X) \rightarrow Z$ be a continuous map.  Assume that for each connected component of the boundary $\del X$, the corresponding subgroup of $\pi_1(Z)$ is an amenable group.  Then for every traversally generic vector field $v$, we have
\[\#\mm(v) \geq \const(n) \cdot \Vert \alpha_*[D(X)] \Vert_{\Delta}.\]
\end{theorem}

That is, the topological quantity $\Vert \alpha_*[D(X)]\Vert_{\Delta}$ is an obstruction to the existence of a traversally generic vector field without maximum-multiplicity trajectories.  In Section~\ref{traversing-setup} we summarize which properties of traversally generic vector fields are needed in order to apply the methods of Gromov from~\cite{Gromov09}.  In Section~\ref{main-section} we present full details for the Amenable Reduction Lemma and Localization Lemma of~\cite{Gromov09}, which are used to prove the ``$\Delta$-Inequality for Generic Maps'' there---Gromov's presentation is rough, so for the reader's convenience we include a full development of the proofs---and then we prove Theorem~\ref{main-theorem} and Theorem~\ref{special-case}.  The new insight of this paper is to bring Gromov's methods to the setting of traversally generic vector fields.

\section{Traversally generic vector fields}\label{traversing-setup}

The purpose of this section is to prove Lemma~\ref{stratification}, which describes the nice properties of a traversally generic vector field and which is a consequence of the work of one of the authors in the paper~\cite{Katz2}.  That paper introduces the definition of traversally generic (Definition 3.2) and proves that the traversally generic vector fields form an open dense subset of the traversing vector fields (Theorem 3.5).  The machinery behind the proof of density comes from the theory of singularities of generic maps, in particular from Thom-Boardman theory (see Theorem 5.2 from Chapter VI of~\cite{Golubitsky74}).  Below, before stating Lemma~\ref{stratification} we give the definitions of traversally generic vector fields and the reduced multiplicity of a trajectory.

The definition of traversally generic includes the notion of boundary generic (Definition 2.1 in ~\cite{Katz2}), which is defined as follows.  Given a traversing vector field $v$ on $X$, we let $\del_2 X$ denote the set of points where $v$ is tangent to $\del X$.  Alternatively, we view $v\vert_{\del X}$ as a section of the normal bundle $TX/T\del X$ of $\del X$ in $X$, and let $\del_2 X$ be the zero locus.  If the section corresponding to $v$ is transverse to the zero section, then $\del_2 X$ is a submanifold of $\del X$ with codimension $1$.  Then we repeat the process using the following iterative construction.  Let $\del_0 X = X$ and $\del_1 X = \del X$.  Once the submanifolds $\del_j X$ have been defined for all $j \leq k$, we view $v\vert_{\del_k X}$ as a section of the normal bundle of $\del_kX$ in $\del_{k-1}X$, and if it is transverse to the zero section, then the zero locus $\del_{k+1}X$ is a submanifold of $\del_k X$ with codimension $1$.  We say $v$ is \textbf{\emph{boundary generic}} if for all $k$, when we view $v\vert_{\del_kX}$ as a section of the normal bundle of $\del_k X$ in $\del_{k-1}X$, this section is transverse to the zero section.

If $v$ is boundary generic, then the \textbf{\emph{multiplicity}} $m(a)$ of any point $a \in X$ is defined to be the greatest $j$ such that $a \in \del_jX$.  By definition, if $m(a) = j > 0$, this means that $v$ is tangent to $\del_{j-1}X$ at $a$ but not tangent to $\del_jX$ there.  Because each $\del_j X$ has dimension $n+1 - j$, the greatest possible multiplicity is $m(a) = n+1$.

Being traversally generic is a property of each trajectory $\gamma$ of $v$.  Using the flow along $v$, we may identify all fibers of the normal bundle of $\gamma$ in $X$; we denote this normal space by $\mathsf{T}_*$.  For each point $a_i \in \gamma \cap \del X$, the tangent space $T\del_{m(a_i)}X$ is transverse to $\gamma$, so it can be viewed as a subspace $\mathsf{T}_i \subseteq \mathsf{T}_*$.  We say that a traversing vector field $v$ is \textbf{\emph{traversally generic}} if $v$ is boundary generic and if for every trajectory $\gamma$ of $v$, the collection of subspaces $\{\mathsf{T}_i(\gamma)\}_i$ is generic in $\mathsf{T}_*$; that is, the quotient map
\[\mathsf{T}_* \rightarrow \bigoplus_{a_i \in \gamma \cap \del X} \mathsf{T}_* / \mathsf{T}_i\]
is surjective.  Recall that the \textbf{\emph{reduced multiplicity}} of every trajectory $\gamma$ is the sum over all $a_i \in \gamma \cap \del X$ of $m(a_i) - 1$.  Thus, because $\dim \mathsf{T}_* = n$ and $\dim T_i = n - (m(a_i) - 1)$, the property of being traversally generic implies $m'(\gamma) \leq n$.

\begin{figure}\label{3-dim}
\begin{center}
\includegraphics[width=3in]{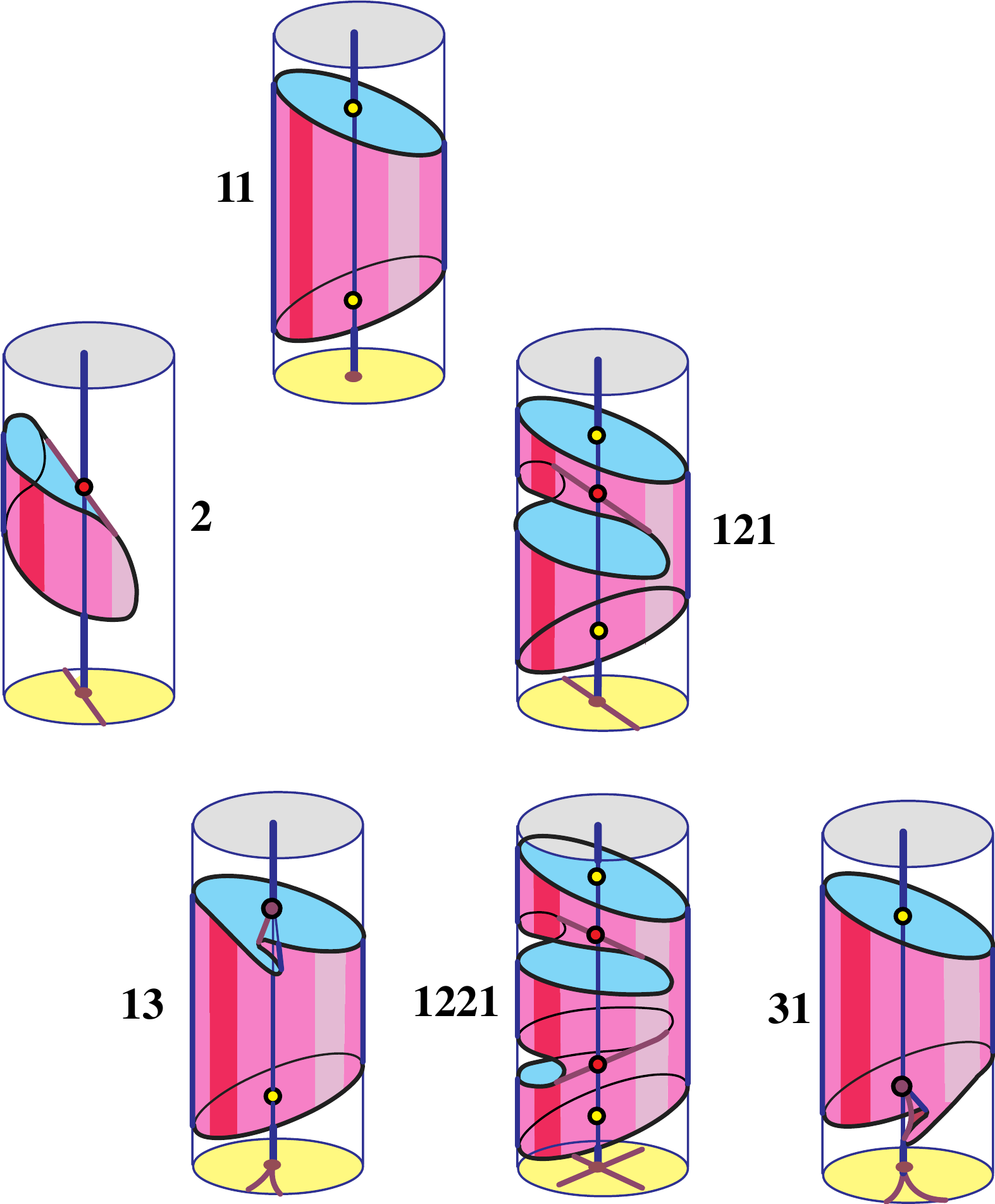}
\end{center}
\caption{When $X$ is $3$-dimensional, there are six possibilities for the local geometry of a trajectory, corresponding to the sequence of multiplicities with which that trajectory meets $\del X$.}
\end{figure}

Lemma~\ref{stratification} describes how every traversally generic vector field $v$ gives rise to stratifications of $\mathcal{T}(v)$ and of $X$; following~\cite{Gromov09} we define a \textbf{\emph{stratification}} of a space to be any partition with the following property: if a stratum $S$ intersects the closure $\overline{S'}$ of another stratum $S'$, then $S \subseteq \overline{S'}$.

\begin{lemma}\label{stratification}
Let $X$ be a compact manifold with boundary, with $\dim X = n + 1$.  The traversally generic vector fields $v$ on $X$ satisfy the following properties:
\begin{enumerate}
\item \label{strat-Y} For $k = 0, \ldots, n$, define $Y_k \subseteq \mathcal{T}(v)$ by
\[Y_k := \{\gamma \in \mathcal{T}(v) : m'(\gamma) = n-k\}.\]
Then every $Y_k$ is a $k$-dimensional manifold, the connected components of all $Y_k$ constitute a stratification of $\mathcal{T}(v)$, and the boundary of each stratum is a union of smaller-dimensional strata.
\item \label{cover} Over each stratum of $\mathcal{T}(v)$, the restriction of $\Gamma\vert_{\del X}$ is a finite covering space map, and the restriction of $\Gamma$ is a trivial bundle map with fiber equal to either an interval or a point.
\item \label{strat-X} For $k = 0, \ldots, n$, define $X_k^\del \subseteq \del X$ by
\[X_k^\del := \Gamma^{-1}(Y_k) \cap \del X,\]
and for $k = 1, \ldots, n+1$, define $X_k^\circ \subseteq X \setminus \del X$ by 
\[X_k^\circ := \Gamma^{-1}(Y_{k-1}) \setminus \del X.\]
Then every $X_k^\del$ and $X_k^\circ$ is a $k$-dimensional submanifold, the connected components of all $X_k^\del$ and $X_k^\circ$ constitute a stratification of $X$, and the boundary of each stratum is a union of smaller-dimensional strata.
\item \label{models} There is a finite collection, depending only on the dimension $n+1$ and not on $v$ or $X$, of stratified local models covering $X$.  That is, each local model is an $(n+1)$-dimensional stratified space with finitely many strata, and every point in $X$ has a neighborhood diffeomorphic to one of the local models in a way that preserves the stratification.
\end{enumerate}
\end{lemma}

The paper~\cite{Katz2} proves (Theorem~3.1) an equivalent characterization of traversally generic vector fields, called versal vector fields (Definition 3.5); the main ingredient in the proof is the Malgrange preparation theorem (see, for instance, Theorem~2.1 from Chapter~IV of~\cite{Golubitsky74}).  When we use the description of versal vector fields, our Lemma~\ref{stratification} becomes straightforward.  In the remainder of the section, we define versal vector fields and explain why they satisfy the properties in Lemma~\ref{stratification}.

For a vector field to be versal means that in a neighborhood of each trajectory there are local coordinates of a certain form.  Figure~\ref{3-dim} depicts the local geometry of these neighborhoods in the case $\dim X = 3$.  In preparation for the definition of a versal vector field, we first define local coordinates near one point.  For each $m$ with $1 \leq m \leq n+1$, we use variables $u \in \mathbb{R}$, $\vec{x} = (x_0, \ldots, x_{m-2}) \in \mathbb{R}^{m-1}$, and $\vec{y} \in \mathbb{R}^{n-(m-1)}$, and define
\[P_m(u, \vec{x}) = u^m + \sum_{\ell = 0}^{m-2} x_\ell u^\ell = u^m + x_{m-2}u^{m-2} + \cdots + x_1u + x_0.\]
We consider the vector field $\frac{\del}{\del u}$ on the space 
\[X_+ = \{(u, \vec{x}, \vec{y}) : P_m(u, \vec{x}) \geq 0\}\]
or on the space
\[X_- = \{(u, \vec{x}, \vec{y}) : P_m(u, \vec{x}) \leq 0\}.\]
The trajectories above each fixed $(\vec{x}, \vec{y})$ stretch between the roots of $P_m(u, \vec{x})$ as a function of $u$.  If $m$ is odd, then $X_+$ has unbounded trajectories in the positive direction (that is, $u \rightarrow +\infty$), and $X_-$ has unbounded trajectories in the negative direction ($u \rightarrow -\infty$).  If $m$ is even, then $X_+$ has unbounded trajectories in both directions, and $X_-$ has only bounded trajectories.  In particular, if $m$ is even, then the trajectory in $X_+$ through the point $(u, \vec{x}, \vec{y}) = 0$ is only that one point.  The vector fields in these local models are boundary generic, and the multiplicity of each point $(u, \vec{x}, \vec{y})$ in the sense defined earlier is equal to the multiplicity of vanishing of $P_m(u, \vec{x})$ as a function of $u$.

A \textbf{\emph{versal vector field}} is described by local coordinates in a neighborhood of each trajectory $\gamma$, as follows.  Suppose $\gamma$ enters $X$ at $a_1 \in \del X$, and then meets $\del X$ at $a_2, \ldots, a_p \in \del X$, in order, exiting at $a_p$.  For each $i$ with $1 \leq i \leq p$, let $\vec{x}_i$ denote a variable in $\mathbb{R}^{m(a_i) - 1}$, and let $\vec{y}$ denote a variable in $\mathbb{R}^{n-m'(\gamma)}$.  Then the coordinates are
\[(u, \vec{x}_1, \ldots, \vec{x}_p, \vec{y}) \in \mathbb{R}^{n+1},\]
and $X$ corresponds to the subset
\[\{(u, \vec{x}_1, \ldots, \vec{x}_p, \vec{y}) : P_{m(a_i)}(u-i, \vec{x}_i) \geq 0 \ \forall i < p,\ P_{m(a_p)}(u-p, \vec{x}_p) \leq 0\},\]
and $v$ corresponds to the vector field $\frac{\del}{\del u}$.  The trajectory $\gamma$ corresponds to the line $(\vec{x}_1, \ldots, \vec{x}_p, \vec{y}) = 0$, and the points $a_1, \ldots, a_p$ correspond to the points $u = 1, \ldots, p$.  Note that for $X$ to be nonempty, we must have either $p = 1$ and $m(a_p)$ is even, or $p > 1$ and both $m(a_1)$ and $ m(a_p)$ are odd while all other $m(a_i)$ are even.  

\begin{proof}[Proof of Lemma~\ref{stratification}]
Because every traversally generic vector field is versal (Theorem 3.1 of~\cite{Katz2}), it suffices to check Lemma~\ref{stratification} for the versal vector fields.  Part~\ref{models} is immediate: there is one local model for each sequence of multiplicities corresponding to reduced multiplicity at most $n$.  Parts~\ref{strat-Y} and~\ref{strat-X} follow from examining the local models: near each trajectory $\gamma$, the only trajectories with reduced multiplicity equal to $m'(\gamma)$ are those with all coordinates $\vec{x}_i$ equal to $0$ (with $\vec{y}$ and $u$ varying).  To prove Part~\ref{cover}, we see from the local models that $\Gamma$ is a locally trivial bundle map over each stratum of $\mathcal{T}(v)$, with fiber equal to either an interval or a point.  Then, because the interval is oriented, and every bundle of oriented intervals is trivial, the bundle must be trivial.
\end{proof}

\section{Main theorem}\label{main-section}

Let $Z$ be a topological space, and $c \in C_*(Z)$ be a singular cycle.  We use the following general strategy to find an upper bound for the simplicial norm $\Vert [c] \Vert_\Delta$: first we generate a large number of cycles homologous to $c$, all with the same norm and with many simplices in common (but with different signs).  Then we take the average of these cycles; the result is homologous to $c$, and because of the cancellation it has small norm.

In order for this cancellation to be possible, we use the simplex-straightening technique from the case where $Z$ is a complete hyperbolic manifold, but we apply a slight generalization for the case where $Z$ is any space with contractible universal cover---that is, $Z = K(\pi_1(Z), 1)$.

\begin{lemma}[\cite{Gromov82}, p.~48]
Let $Z$ be a space with contractible universal cover $\widetilde{Z}$.  Then there is a ``straightening'' operator 
\[\straight : C_*(Z) \rightarrow C_*(Z)\] 
with the following properties:
\begin{itemize}
\item For each simplex $\sigma \in C_*(Z)$, the straightened version $\straight(\sigma)$ is a simplex of the same dimension with the same sequence of vertices.
\item If two simplices $\sigma_1, \sigma_2 \in C_*(Z)$ have the same sequence of vertices, and their lifts $\widetilde{\sigma}_1, \widetilde{\sigma}_2 \in C_*(\widetilde{Z})$ to the universal cover also have the same sequence of vertices, then $\straight(\sigma_1) = \straight(\sigma_2)$.
\item $\straight$ commutes with the boundary map $\del$; that is, $\straight$ is a chain-complex endomorphism.
\item $\straight$ commutes with the standard action of the symmetric group $S_{j+1}$ on each $C_j(Z)$.
\item $\straight$ is chain homotopic to the identity.
\end{itemize}
\end{lemma}

\begin{proof}
We construct the straightening operator and the chain homotopy simultaneously, one dimension at a time.  The chain homotopy will be, for each simplex $\sigma : \Delta \rightarrow Z$, a map $H(\sigma) : \Delta \times I \rightarrow Z$ such that the restriction of $H(\sigma)$ to $\Delta \times 0$ is $\sigma$, the restriction to $\Delta \times 1$ is $\straight(\sigma)$, and for each face $F \in \del \Delta$, the restriction of $H(\sigma)$ to $F \times I$ is $H(\sigma \vert_F)$.

For every $0$-dimensional simplex $\sigma^0 \in C_0(Z)$, we have $\straight(\sigma^0) = \sigma^0$ and a constant homotopy $H(\sigma^0)$.  For a simplex $\sigma$ with $\dim \sigma = j > 0$, suppose the straightening and chain homotopy are already defined for every dimension less than $j$.  In particular, $\straight(\del \sigma)$ depends only on the sequence of vertices of the lift $\widetilde{\sigma}$ of $\sigma$ to the universal cover $\widetilde{Z}$.  We lift $\straight(\del \sigma)$ to $\widetilde{Z}$; because $\widetilde{Z}$ is contractible, there is some simplex filling in the lift of $\straight(\del \sigma)$, and we can choose $\straight(\sigma)$ to be the corresponding simplex in $Z$.  We make this choice only once per orbit of $\pi_1(Z) \times S_{j+1}$ on the set $(\widetilde{Z})^{j+1}$ of sequences of vertices in $\widetilde{Z}$.

Having chosen $\straight(\sigma)$, we lift $\sigma$, $\straight(\sigma)$, and $H(\del \sigma)$ to $\widetilde{Z}$ to form a sphere of dimension $j$.  Because $\widetilde{Z}$ is contractible, we can fill in this sphere by a map $\widetilde{H(\sigma)}$ on $\Delta \times I$ that has the prescribed boundary, and let $H(\sigma)$ be the corresponding map into $Z$.
\end{proof}

We also use the anti-symmetrization operator,
\[\symm : C_*(Z) \rightarrow C_*(Z)\]
given by
\[\symm(\sigma^j) = \frac{1}{(j+1)!} \sum_{q \in S_{j+1}} \sign(q) \cdot q(\sigma^j)\]
for every simplex $\sigma^j \in C_j(Z)$.  Gromov states that this operator is chain homotopic to the identity (\cite{Gromov82}, p.~29), and Fujiwara and Manning give the proof in~\cite{Fujiwara11}.

\begin{lemma}[\cite{Fujiwara11}, Appendix B]
Let $Z$ be any topological space.  The anti-symmetrization operator $\symm : C_*(Z) \rightarrow C_*(Z)$ is chain homotopic to the identity.
\end{lemma}

Thus, the composition of these two operators 
\[\symm \circ \straight : C_*(Z) \rightarrow C_*(Z),\] 
satisfies the following properties:
\begin{enumerate}
\item $\symm \circ \straight(\sigma)$ depends only on the list of vertices of the lift $\widetilde{\sigma}$ to the universal cover. \label{vertices}
\item For every $q \in S_{j+1}$ and every $\sigma^j \in C_j(Z)$, we have
\[\symm \circ \straight \circ\,q(\sigma^j) = \sign(q) \cdot \symm \circ \straight(\sigma^j).\] \label{transposition}
\item $\symm \circ \straight$ is a chain map, chain homotopic to the identity.
\item $\symm \circ \straight$ does not increase the norm; that is, for every chain $c \in C_*(Z)$, we have
\[\Vert \symm \circ \straight(c) \Vert_\Delta \leq \Vert c \Vert_\Delta.\]
\end{enumerate}

Property~\ref{transposition} is our reason for introducing $\symm$ at all: it allows homotopic simplices with opposite orientations to cancel in a sum or average.

Below we state the setup for the next lemma.  Let $X$ be a topological space, and let $c \in C_*(X)$ be a singular cycle.  We view $c$ as a triple $(\Sigma, c_\Sigma, f)$ where $\Sigma$ is a simplicial complex, $c_\Sigma$ is a simplicial cycle on $\Sigma$, and $f : \Sigma \rightarrow X$ is a continuous map such that $c = f_*c_\Sigma$.  

By a \textbf{\emph{partial coloring}} of $c$ we mean a list $V_1, V_2, \ldots, V_\ell, \ldots$ of disjoint subsets of the set of vertices of $\Sigma$.  According to the partial coloring we classify each simplex as either essential or non-essential; the non-essential simplices are the ones that can be made to disappear in a certain sense.  Accordingly, we define essential simplex by defining non-essential simplex, as follows.  A simplex $\sigma$ of $\Sigma$ is called an \textbf{\emph{essential simplex}} of $c$ if neither of the following conditions holds:
\begin{itemize}
\item $\sigma$ has two distinct vertices in the same $V_\ell$; or
\item $\sigma$ has two vertices that are the same point of $\Sigma$, and the edge between them is a null-homotopic loop in $X$.
\end{itemize}
Let $Z$ be a space with contractible universal cover and let $\alpha : X \rightarrow Z$ be a continuous map that sends all vertices of $c$ in $X$ to the same point of $Z$.  Let $\Gamma_\ell$ denote the subgroup of $\pi_1(Z)$ generated by the $\alpha$-images of the edges of $c$ for which both endpoints are in $V_\ell$.

\begin{lemma}[Amenable Reduction Lemma, p.~25 of~\cite{Gromov09}]\label{amen-red}
Suppose that $\Gamma_\ell$ is an amenable group for every $\ell$.  Then the simplicial norm of the $\alpha$-image of the homology class $[c] \in H_*(X)$ represented by the cycle $c = \sum r_i \sigma_i$ (where $r_i \in \mathbb{R}$ are coefficients and $\sigma_i$ are simplices) satisfies
\[\Vert \alpha_*[c]\Vert_{\Delta} \leq \sum_{\sigma_i \textrm{ essential}} \abs{r_i}.\]
\end{lemma}

\begin{proof}
Given a singular simplex $\sigma \in C_*(X)$ and a path $\gamma : [0, 1] \rightarrow X$ beginning at one vertex of $\sigma$, there is a homotopy $\sigma_t$ pushing the vertex along $\gamma$; the image of each $\sigma_t$ is the union of the image of $\sigma$ with the partial path $\gamma \vert_{[0, t]}$.  Given a singular cycle $c \in C_*(X)$ and a path $\gamma$ beginning at one vertex of $c$, we may apply this process to every simplex of $c$ containing that vertex, to obtain a homotopic (and thus homologous) cycle $\gamma * c$.  More precisely, if $c$ is $(\Sigma, c_\Sigma, f)$ then we modify $f$ by a homotopy supported in a neighborhood of one vertex of $\Sigma$.  Likewise, we may take a path $\gamma$ in $Z$ rather than in $X$, and obtain a cycle $\gamma * \alpha_*c$, for which the straightened cycle $\straight(\gamma * \alpha_*c)$ depends on $\gamma$ only up to homotopy.  That is, if $c$ is $(\Sigma, c_\Sigma, f)$ then $\alpha_*c$ is $(\Sigma, c_\Sigma, \alpha \circ f)$ and we homotope $\alpha \circ f$.

Applying this process to every vertex of $\Sigma$ in $\bigcup_\ell V_\ell$ simultaneously, we obtain an action of the product group $\times_\ell (\Gamma_\ell)^{V_\ell}$ on $c$, given by
\[g \mapsto \symm \circ \straight(g * \alpha_*c),\ g \in \times_\ell (\Gamma_\ell)^{V_\ell}.\]
That is, if $g = (\gamma_v)_{v \in \bigcup_\ell V_\ell}$, then to find $g * \alpha_*c$ we choose disjoint neighborhoods in $\Sigma$ of all vertices $v \in \bigcup_\ell V_\ell$ and for each $v$ we homotope $\alpha \circ f$ in the chosen neighborhood of $v$ to push $\alpha \circ f(v)$ along $\gamma_v$ in $Z$.  We will take the average of cycles $\symm \circ \straight(g * \alpha_*c)$ as $g$ ranges over a large finite subset of $\times_\ell (\Gamma_\ell)^{V_\ell}$.  To choose this subset, we use the definition of amenable group.

One characterization of (discrete) amenable groups is the F\o lner criterion: for every amenable group $\Gamma$, every finite subset $S \subset \Gamma$, and every $\varepsilon > 0$, there is a finite subset $A \subset \Gamma$ satisfying the inequality
\[\frac{\abs{xA\, \Delta\, A}}{\abs{A}} \leq \varepsilon\ \ \forall x \in S,\]
where $\Delta$ denotes the symmetric difference.  In our setting, we choose $S_\ell$ to be the set of $\alpha$-images of edges in $c$ with both endpoints in $V_\ell$, and then apply the F\o lner criterion to find $A_\ell \subset \Gamma_\ell$.  We take the average of $\symm \circ \straight(g * \alpha_*c)$ for $g \in \times_\ell (A_\ell)^{V_\ell} \subset \times_{\ell} (\Gamma_\ell)^{V_\ell}$; the result is some cycle homologous to $\alpha_*c$ which we show has small norm.

First we show that if $\sigma$ is \emph{not} an essential simplex of $c$, then the average of $\symm \circ \straight(g * \alpha_*\sigma)$ has norm at most $\varepsilon$.  If one edge of $\sigma$ is a contractible loop, then every $\symm \circ \straight(g * \alpha_*\sigma)$ is equal to $0$ (using properties~\ref{vertices} and~\ref{transposition} of $\symm \circ \straight$), so the average is $0$.  Thus, we address the case where $\sigma$ has two distinct vertices $v_1$ and $v_2$ in some $V_\ell$.  When averaging over all $g \in \times_\ell (A_\ell)^{V_\ell}$, we average separately over each slice where only the $v_1$ and $v_2$ components of $g$ vary and all other components are fixed.  It suffices to show that the average of $\symm \circ \straight(g * \alpha_* \sigma)$ over each such slice $A_\ell \times A_\ell$ has norm at most $\varepsilon$.  

Let $\gamma_1, \gamma_2 \in A_\ell$ denote the $v_1$ and $v_2$ components of $g$, which we write as $g = g_{(\gamma_1, \gamma_2)}$, and let $x \in \Gamma_\ell$ denote the $\alpha$-image of the edge in $c$ between $v_1$ and $v_2$.  The edge $x$ in $\alpha_*\sigma$ becomes an edge $\gamma_1^{-1}x\gamma_2$ in $g_{(\gamma_1, \gamma_2)} * \alpha_*\sigma$.  Consider the involution 
\[(\gamma_1, \gamma_2) \mapsto (x\gamma_2, x^{-1}\gamma_1)\]
on the square subset
\[(\gamma_1, \gamma_2) \in (xA_\ell \cap A_\ell) \times (x^{-1}A_\ell \cap A_\ell) \subset A_\ell \times A_\ell.\]
The path resulting from $(x\gamma_2, x^{-1}\gamma_1)$ is the inverse of the path resulting from $(\gamma_1, \gamma_2)$, and thus (using property~\ref{transposition} of $\symm \circ \straight$) we have
\[\sum_{(\gamma_1, \gamma_2) \in (xA_\ell \cap A_\ell) \times (x^{-1}A_\ell \cap A_\ell)} \symm \circ \straight(g_{(\gamma_1, \gamma_2)} * \alpha_*\sigma) = 0.\]
In other words, only those $(\gamma_1, \gamma_2)$ outside the square subset contribute to the average.  By the F\o lner criterion we have
\[\abs{ xA_\ell \cap A_\ell} \geq \left(1 - \frac{\varepsilon}{2}\right) \abs{A_\ell},\]
and so
\[\abs{ (xA_\ell \cap A_\ell) \times (x^{-1}A_\ell \cap A_\ell)} \geq \left( 1-\frac{\varepsilon}{2} \right)^2 \abs{A_\ell \times A_\ell} > (1-\varepsilon) \abs{A_\ell \times A_\ell}.\]
Thus the average over $A_\ell \times A_\ell$ satisfies
\[\frac{1}{\abs{A_\ell \times A_\ell}} \sum_{(\gamma_1, \gamma_2) \in A_\ell \times A_\ell} \symm \circ \straight(g_{(\gamma_1, \gamma_2)} * \alpha_*\sigma) \leq \]
\[\leq \frac{1}{\abs{A_\ell \times A_\ell}} \sum_{(\gamma_1, \gamma_2) \not\in (xA_\ell \cap A_\ell) \times (x^{-1}A_\ell \cap A_\ell)} 1 < \varepsilon.\]

Taking the sum over all simplices $\sigma_i$ of $c$, we obtain the inequality
\[\Vert \alpha_*[c] \Vert_\Delta \leq \sum_{\sigma_i \text{ essential}} \abs{r_i} + \sum_{\sigma_i \text{ not essential}} \varepsilon\abs{r_i},\]
and taking the limit as $\varepsilon \rightarrow 0$ we obtain the inequality of the lemma statement.
\end{proof}

Recall the definition of stratification: if a stratum $S$ intersects the closure $\overline{S'}$ of another stratum $S'$, then $S \subseteq \overline{S'}$.  In this case we write $S \preceq S'$.  If neither $S \preceq S'$ nor $S' \preceq S$, then we say the two strata are \textbf{\emph{incomparable}}.

The next lemma involves the notion of \textbf{\emph{stratified simplicial norm}} (as with simplicial norm, it is really a semi-norm), which for a homology class $h$ on a stratified space $X$ is the infimum of norms of all cycles $c$ representing $h$ that are consistent with the stratification, in the following sense illustrated in Figure~\ref{strat}.  Gromov gives two conditions: \emph{ord}(er) and \emph{int}(ernality) (\cite{Gromov09}, p.~27).  We use these two conditions plus two more:
\begin{itemize}
\item   We require that for each simplex of $c$, the image of the interior of each face (of any dimension) must be contained in one stratum.  (This condition may be implicit in Gromov's paper.)  We call this the \emph{cellular} condition.
\item The (\emph{ord}) condition states that the image of each simplex of $c$ must be contained in a totally ordered chain of strata; that is, the simplex does not intersect any two incomparable strata.  
\item The (\emph{int}) condition states that for each simplex of $c$, if the boundary of a face (of any dimension) maps into a stratum $S$, then the whole face maps into $S$.  
\item For technical reasons involving the amenable reduction lemma (Lemma~\ref{amen-red}), we require that if two vertices of a simplex of $c$ map to the same point $v \in X$, then the edge between them must be constant at $v$.  We call this the \emph{loop} condition.
\end{itemize}
The stratified simplicial norm of a homology class $h$ on a space $X$ with stratification $\mathcal{S}$ is denoted by $\Vert h \Vert_{\Delta}^\mathcal{S}$.

\begin{figure}
\begin{center}
\includegraphics[width=3in]{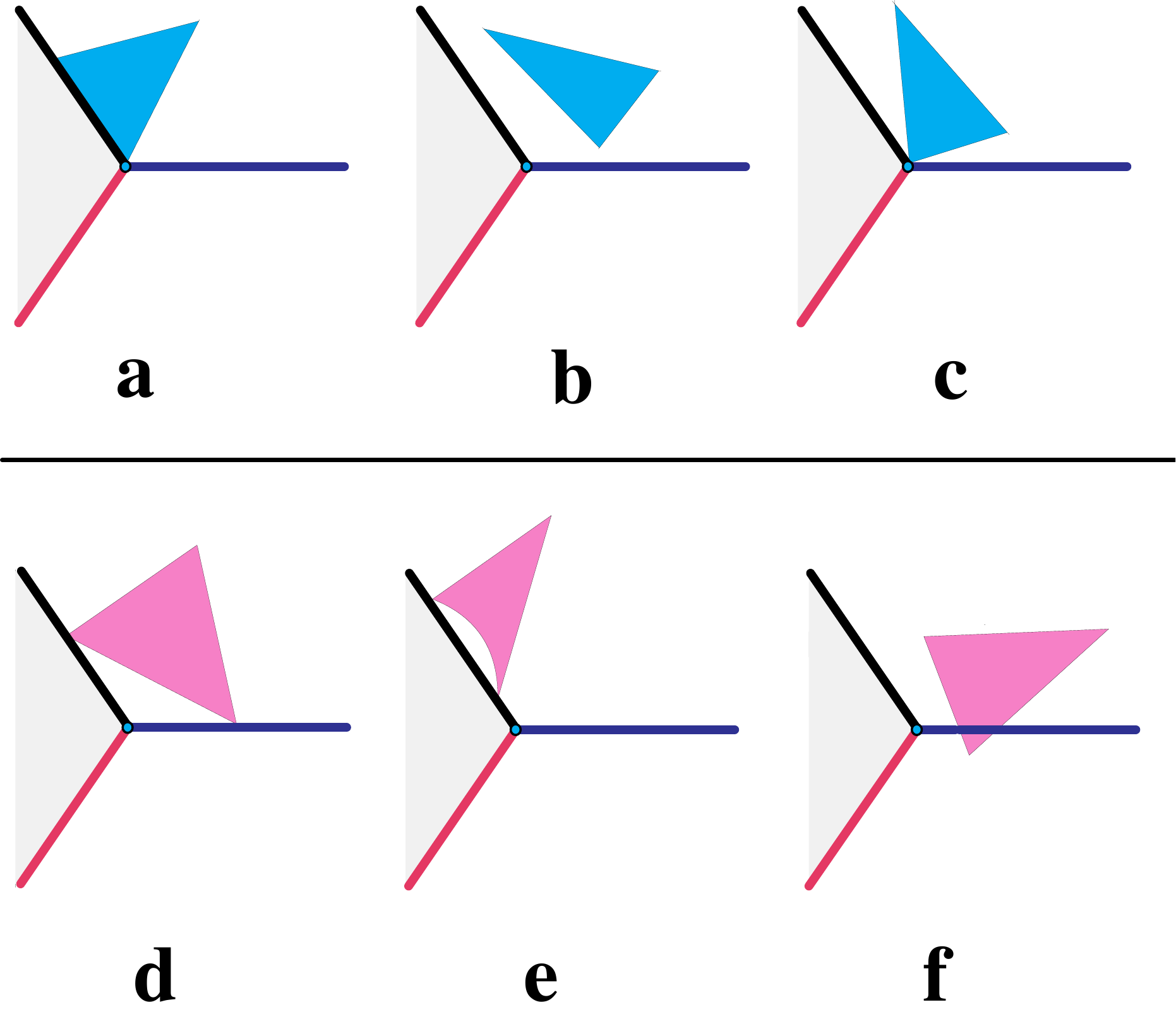}
\end{center}
\caption{An example stratification of the plane: one $0$-dimensional stratum, three $1$-dimensional strata, and three $2$-dimensional strata.  In diagrams (a), (b), and (c), the shaded triangle satisfies all four criteria for simplices used to compute the stratified simplicial norm.  Diagram (d) violates the (\emph{ord}) condition, diagram (e) violates the (\emph{int}) condition, and diagram (f) violates the \emph{cellular} condition.}\label{strat}
\end{figure}

\begin{lemma}[Localization Lemma, p.~27 of~\cite{Gromov09}]\label{localization}
Let $X$ be a closed manifold with stratification $\mathcal{S}$ consisting of finitely many connected submanifolds, and let $j$ be an integer between $0$ and $\dim X$.  Let $Z$ be a space with contractible universal cover, and let $\alpha : X \rightarrow Z$ be a continuous map such that the $\alpha$-image of the fundamental group of each stratum of codimension less than $j$ is an amenable subgroup of $\pi_1(Z)$.  Let $X_{-j} \subseteq X$ denote the union of strata with codimension at least $j$, and let $U$ be a neighborhood of $X_{-j}$ in $X$.  Then the $\alpha$-image of every $j$-dimensional homology class $h \in H_j(X)$ satisfies the bound
\[\Vert \alpha_*h \Vert_{\Delta} \leq \Vert h_U\Vert_{\Delta}^\mathcal{S},\]
where $h_U \in H_j(U, \del U)$ denotes the restriction of $h$ to $U$, obtained from the composite homomorphism 
\[H_j(X) \rightarrow H_j(X, X \setminus U) \rightarrow H_j(U, \del U),\]
where the last map is the excision isomorphism.
\end{lemma}

To prove the localization lemma, we construct a partition of $X$ as follows.

\begin{lemma}
Let $X$ be a closed manifold, with a metric space structure and stratified by finitely many submanifolds.  For every $\varepsilon > 0$, there is a partition of $X$ consisting of one subset $P_S$ for each stratum $S$, and some $\delta > 0$, with the following properties:
\begin{itemize}
\item If $S$ and $S'$ are incomparable strata, then $\dist(P_S, P_{S'}) > \delta$, $\dist(P_S, S') > \delta$, and $\dist(S, P_{S'}) > \delta$.
\item If $x \in P_S$, then $\dist(x, S) < \varepsilon$.
\item Let $N_\delta(P_S)$ denote the $\delta$-neighborhood of $P_S$.  There is a homotopy beginning with the inclusion $N_\delta(P_S) \hookrightarrow X$ and ending with a map with image in $S$.
\end{itemize}
\end{lemma}

\begin{figure}
\begin{center}
\includegraphics[width=3.5in]{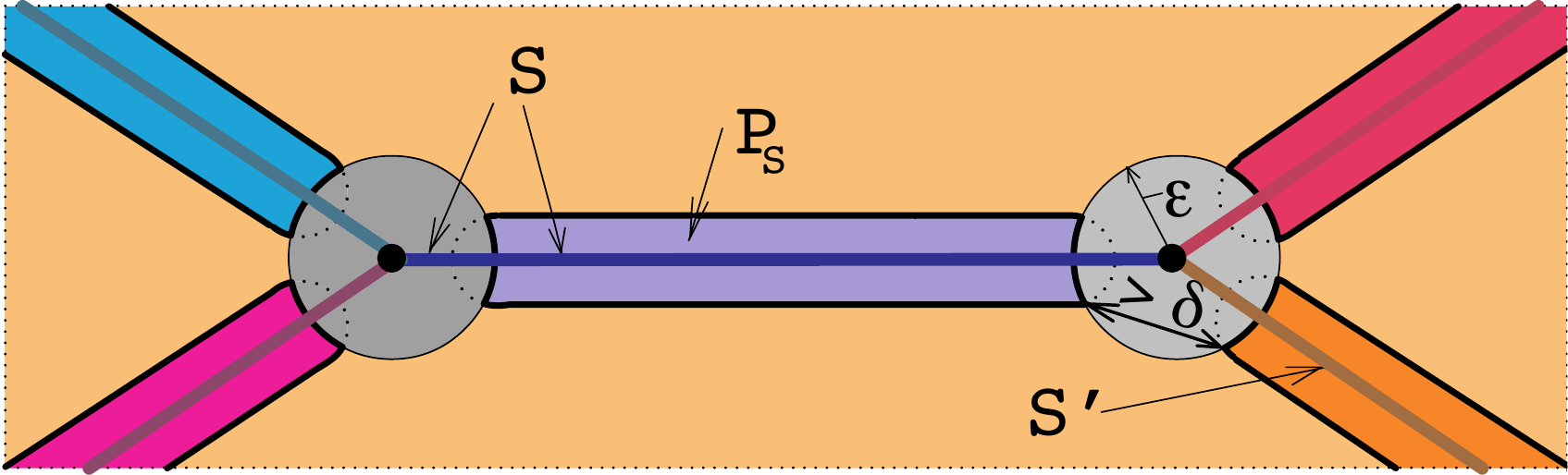}
\end{center}
\caption{For each stratum $S$, its approximation $P_S$ lies within the $\varepsilon$-neighborhood of $S$.  For every two incomparable strata $S$ and $S'$, the distance between the sets $P_S$ and $P_{S'}$ exceeds $\delta = \delta(\varepsilon)$.}\label{part}
\end{figure}

\begin{proof}
Figure~\ref{part} depicts the relationship between the strata $S$ and the subsets $P_S$.  The sets $P_S$ are determined by a choice of small numbers
\[0 < \varepsilon_{\dim X} < \cdots < \varepsilon_0 < \varepsilon,\]
which we choose inductively.  
\begin{itemize}
\item Step 0: We choose $\varepsilon_0 < \varepsilon$ such that for every $0$-dimensional stratum $S_0$, the ball $N_{3\varepsilon_0}(S_0)$ is Euclidean and has the following property: if $S_0$ is disjoint from the closure $\overline{S'}$ of another stratum $S'$ (i.e., if $S_0 \not\preceq S'$), then $N_{3\varepsilon_0}(S_0)$ is also disjoint from $\overline{S'}$.  We put $P_{S_0} = N_{\varepsilon_{0}}(S_0)$.
\item Step 1: First we find a tubular neighborhood $U_{S_1}$ of every $1$-dimensional stratum $S_1$, such that if $S_1$ is disjoint from some $\overline{S'}$, then $U_{S_1}$ is also disjoint from $\overline{S'}$.  The portion of $S_1$ that is outside the union of all $P_{S_0}$ is a compact set.  Therefore, we can choose $\varepsilon_1 < \varepsilon_0$ such that for every $1$-dimensional stratum $S_1$, we have
\[N_{3\varepsilon_1}\left(S_1 \setminus \bigcup_{S_0} P_{S_0}\right) \subseteq U_{S_1}.\]
We put 
\[P_{S_1} = N_{\varepsilon_1}\left(S_1 \setminus \bigcup_{S_0} P_{S_0}\right) \setminus \bigcup_{S_0} P_{S_0}.\]
\item Step $k$: Having chosen $\varepsilon_{k-1} < \cdots < \varepsilon_0 < \varepsilon$, we choose $\varepsilon_k < \varepsilon_{k-1}$ much as in Step 1.  We find a tubular neighborhood $U_{S_k}$ of every $k$-dimensional stratum $S_k$, such that if $S_k$ is disjoint from some $\overline{S'}$, then $U_{S_k}$ is also disjoint from $\overline{S'}$.  We choose $\varepsilon_k < \varepsilon_{k-1}$ such that for every $k$-dimensional stratum $S_k$, we have
\[N_{3\varepsilon_k}\left(S_k \setminus \bigcup_{i = 0}^{k-1}\bigcup_{S_i} P_{S_i}\right) \subseteq U_{S_k}.\]
We put
\[P_{S_k} = N_{\varepsilon_k}\left(S_k \setminus \bigcup_{i = 0}^{k-1}\bigcup_{S_i} P_{S_i}\right) \setminus \bigcup_{i = 0}^{k-1}\bigcup_{S_i} P_{S_i}.\]
\item Step $\dim X$: Formally, the procedure from Step $k$ applies.  However, we note that each tubular neighborhood $U_{S_{\dim X}}$ is equal to all of $S_{\dim X}$, and so we have
\[P_{S_{\dim X}} = S_{\dim X} \setminus \bigcup_{i = 0}^{\dim X - 1} \bigcup_{S_i} P_{S_i}.\]
We do choose $\varepsilon_{\dim X}$ just as in Step $k$.
\end{itemize}
We choose $\delta < \varepsilon_{\dim X}$.  The second and third properties in the lemma statement are immediate.  For the first property, suppose $S$ and $S'$ are incomparable.  In particular $S \not\preceq S'$.  Then all of $S'$ lies at least $3\varepsilon_{\dim S}$ away from $S \cap P_S$, whereas all of $P_S$ lies within $\varepsilon_{\dim S}$ of $S \cap P_S$, and so we have
\[\dist(P_S, S') \geq 2\varepsilon_{\dim S} > \delta,\]
and likewise with $S$ and $S'$ reversed.  To check $\dist(P_S, P_{S'})$, assume without loss of generality $\dim S \leq \dim S'$.  Then we have
\[\dist(P_S, P_{S'}) \geq 2\varepsilon_{\dim S} - \varepsilon_{\dim S'} \geq \varepsilon_{\dim S} > \delta.\]
\end{proof}

\begin{proof}[Proof of Lemma~\ref{localization}]
Here is the rough idea of the proof: we extend each relative cycle $c_U$ representing $h_U$ to a cycle $c$ representing $h$.  We construct a partial coloring on the vertices of $c$ with one subset $V_\ell$ for each stratum $S_\ell$ of codimension less than $j$.  If $c$ is chosen carefully, then every new simplex of $c$ is not essential, so the amenable reduction lemma (Lemma~\ref{amen-red}) implies that the new simplices do not contribute to the simplicial norm of $\alpha_*h$.

In fact, the proof gets more complicated because we need to guarantee that the $\alpha$-images of the edges of $c$ with both endpoints in $V_\ell$ generate a subgroup of $\alpha_*\pi_1(S_\ell)$, and thus an amenable subgroup of $\pi_1(Z)$.  (Every subgroup of an amenable group is amenable.)  We construct the partition $\{P_S\}$ with $\varepsilon$ chosen to be smaller than the distance from $X_{-j}$ to the complement of $U$, and use the $\delta$ arising from the construction of $\{P_S\}$.  Then we construct chains $c_1$, $c_2$, and $c_\delta$ so that $c = c_U + c_1 + c_2 + c_\delta$ is a cycle homologous to $h$, by the following method depicted in Figure~\ref{loc}.

\begin{figure}
\begin{center}
\includegraphics[width=3.5in]{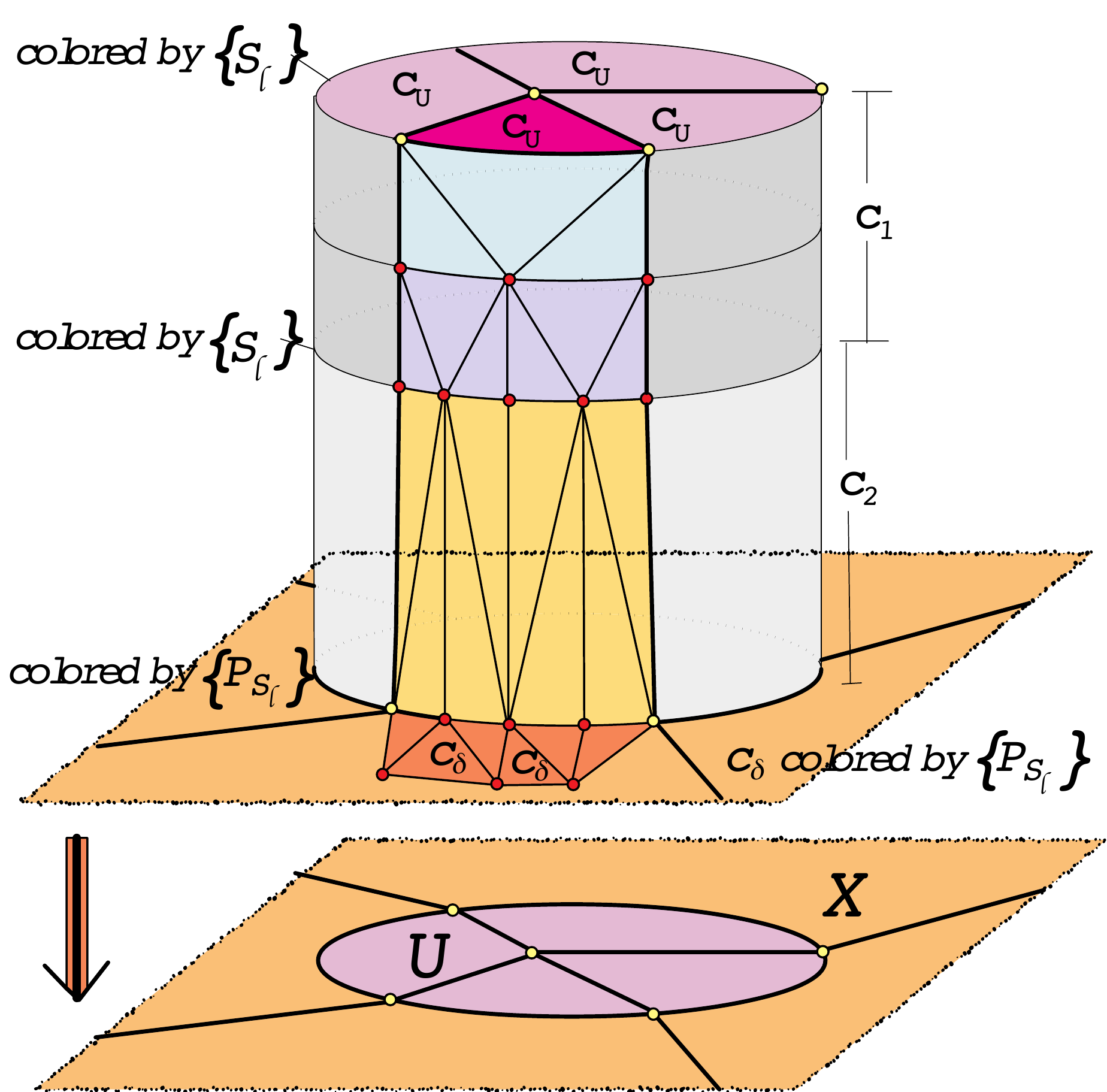}
\end{center}
\caption{The singular cycle $c$ is constructed as the sum $c = c_U + c_1 + c_2 + c_\delta$, where $c_U$ is the restriction to $U$, and $c_\delta$ is the restriction to the complement of $U$, and the cylinders $c_1$ and $c_2$ connect $c_U$ to $c_\delta$.  The vertices of $c_U$ and $c_1$ are colored according to which stratum $S_\ell$ they are in, the vertices of $c_\delta$ are colored according to which subset $P_{S_\ell}$ they are in, and the cylinder $c_2$ connects the two coloring methods.}\label{loc}
\end{figure}

\begin{itemize}
\item To construct $c_\delta$, we first take a cycle $c'$ representing $h$ such that $c_U$ is its restriction to $U$, and then obtain $c_\delta$ by iterated barycentric subdivision of $c' - c_U$ so that the diameter of each simplex is less than $\delta$.  For each vertex $v$ of $c_\delta$, if $v \in P_{S_\ell}$, then $v \in V_\ell$.
\item $c_2$ is the cylinder $-\del c_\delta \times [0, 1]$, triangulated in such a way that no new vertices are created, and mapped to $X$ by the projection $-\del c_\delta \times [0, 1] \rightarrow -\del c_\delta$.  The vertices of $-\del c_\delta \times 1$ are identified with the vertices of $c_\delta$, so their partial coloring is determined by their membership in $P_{S_\ell}$.  The vertices of $-\del c_\delta \times 0$ have a different partial coloring: if $v \in S_\ell$, then $v \in V_\ell$.
\item $c_1$ is a subdivision of the cylinder $\del c_U \times [0, 1]$, mapped to $X$ by the projection $\del c_U \times [0, 1] \rightarrow \del c_U$.  The end $\del c_U \times 0$ is identified with $\del c_U$ and is not subdivided.  The end $\del c_U \times 1$ is divided by barycentric subdivision so that it may be identified with the $0$ end of $c_2$, which is equal to $-\del c_\delta$.  The middle of the cylinder is subdivided by concatenating the chain homotopies corresponding to barycentric subdivision, one for each iteration.  For each vertex $v$ of $c_1$, if $v \in S_\ell$, then $v \in V_\ell$.
\end{itemize}

First we verify that every simplex in $c_1$, $c_2$, and $c_\delta$ is not essential.  The (\emph{ord}) condition on $\del c_U$ and the choice of $\delta$ imply that the labels of the $(j+1)$ vertices of each simplex correspond to $S_\ell$ in a totally ordered chain of strata.  Because each $S_\ell$ has codimension between $0$ and $j-1$, and every two strata of the same dimension are incomparable, two of the vertices must have the same label.  If these two vertices are identical in $c$, then the edge between them must be constant; this results from the \emph{loop} property on $c_U$ and the fact that barycentric subdivision destroys loops.

Next we need to check that the $\alpha$-images of the edges with both endpoints in each subset $V_\ell$ generate an amenable subgroup of $\pi_1(Z)$.  In the current setup, the $\alpha$-images of these edges are not even loops in $Z$.  We correct this problem by modifying $c$ by a homotopy that adds a path to each vertex, as in the proof of Lemma~\ref{amen-red}; the path is chosen as follows.  For each stratum $S_\ell$, we choose one special point $x_\ell \in S_\ell$.  We homotope $c$ so that every vertex $v \in V_\ell$ travels along some path ending at $x_\ell$, chosen as follows:

\begin{itemize}
\item If $v \in c_\delta$ and $v \in V_\ell$, then $v \in P_{S_\ell}$, so we choose the path to be the trajectory of $v$ under the homotopy sending $N_\delta(P_{S_\ell})$ into $S_\ell$.  Then we concatenate this path with any path contained in $S_\ell$ and ending at $x_\ell$.
\item If $v$ is in the $0$ end of $c_2$, and $v \in V_\ell$, then $v \in S_\ell$.  If $v$ has an edge to some vertex in the $1$ end of $c_2$ (and thus in $c_\delta$) that is in $P_{S_\ell}$, then $v \in N_\delta(P_{S_\ell})$, so as above we take the trajectory of $v$ under the homotopy sending $N_\delta(P_{S_\ell})$ into $S_\ell$.  If $v$ does not have such an edge, then we take a constant path at $v$ instead.  Then we concatenate this first path (either of the two options) with any path contained in $S_\ell$ and ending at $x_\ell$.
\item If $v$ is any other vertex in $c_1$---that is, not in the $1$ end---and $v \in V_\ell$, then we take a path contained in $S_\ell$ from $v$ to $x_\ell$.
\end{itemize}
Then we homotope $\alpha$ (or $c$ again) so that the image of every $x_\ell$ is the same point in $Z$.  Now the $\alpha$-images induced by a given $V_\ell$ do generate a subgroup of $\pi_1(Z)$.  In order to show that this subgroup is amenable, it suffices to show that it is contained in the amenable subgroup $\alpha_*i_*\pi_1(S_\ell)$, where $i : S_\ell \hookrightarrow X$ denotes the inclusion.  Every edge $\gamma$ with both endpoints in $V_\ell$ is a loop at $x_\ell$; we show that its homotopy class $[\gamma]$ in $\pi_1(X, x_\ell)$ is in $i_*\pi_1(S_\ell)$---that is, $\gamma$ is homotopic through loops at $x_\ell$ to a loop entirely contained in $S_\ell$.  Then $\alpha_*[\gamma] \in \alpha_*i_*\pi_1(S_\ell)$.  

We construct the homotopy on $\gamma$ as follows.  If $\gamma$ is in $c_2$ or $c_\delta$, and at least one endpoint is in $P_{S_\ell}$, then all of $\gamma$ is in $N_\delta(P_{S_\ell})$, so we homotope $\gamma$ by the homotopy sending $N_\delta(P_{S_\ell})$ into $S_\ell$.  If $\gamma$ is in $c_2$ or $c_1$, and both endpoints are in $S_\ell$, then we use the fact that the \emph{cellular} property and the (\emph{int}) property together are preserved by barycentric subdivision.  Thus $\gamma$ is already contained in $S_\ell$.

By this method, we produce a cycle $c$, extending $c_U$ and homotopic to $h$, with a partial coloring such that every new simplex is not essential and such that, after a homotopy of $\alpha$, every edge of $c$ with both endpoints in $V_\ell$ is a loop representing an element of $\alpha_*\pi_1(S_\ell)$.  Applying the amenable reduction lemma (Lemma~\ref{amen-red}), we obtain
\[\Vert \alpha_* h \Vert_{\Delta} \leq \Vert c_U \Vert_\Delta,\]
and taking the infimum over all such $c_U$,
\[\Vert \alpha_* h \Vert_{\Delta} \leq \Vert h_U \Vert_{\Delta}^\mathcal{S}.\]
\end{proof}

\begin{proof}[Proof of Theorem~\ref{main-theorem}]
From Part~\ref{strat-X} of Lemma~\ref{stratification}, the vector field $v$ gives rise to a stratification of $X$; doubling this stratification produces a stratification of the closed manifold $D(X)$ by submanifolds.  From Part~\ref{models} of Lemma~\ref{stratification}, there are only finitely many strata, because the compact set $X$ can be covered by finitely many neighborhoods each matching one of the local models.

In order to apply the localization lemma (Lemma~\ref{localization}), we need to check that for each stratum $S$ of $X$, the subgroup $\alpha_*\pi_1(S)$ of $\pi_1(Z)$ is an amenable group.  We have assumed that this is true if $S \subseteq \del X$.  (Every subgroup of an amenable group is amenable.)  Otherwise, we apply Parts~\ref{cover} and~\ref{strat-X} of Lemma~\ref{stratification}: $S$ is one connected component of $\Gamma^{-1}(\sigma) \setminus \del X$ for some stratum $\sigma$ of $\mathcal{T}(v)$, and the entire preimage $\Gamma^{-1}(\sigma)$ is a trivial bundle $\sigma \times F$, for some fiber $F$.  Under the stratification of $X$, the fiber $F$ is an interval subdivided by finitely many points from $\del X$.  There is a homotopy on the $1$-dimensional part of $F$ that pushes each open subinterval to the next point of $\del X$, which gives a homotopy on $\Gamma^{-1}(\sigma) \setminus \del X$ that starts with the inclusion into $X$ and ends with a map into $\del X$.  Applying this homotopy to loops in $S$ we see that $\pi_1(S)$ is contained in $\pi_1(\del X)$, so its $\alpha$-image is an amenable group.

Now we apply the localization lemma (Lemma~\ref{localization}), with $j = \dim X = n + 1$.  Then $X_{-j}$ consists of the $0$-dimensional strata, which are the intersections of the maximum-multiplicity trajectories with $\del X$.  Let $x_1, \ldots, x_r$ denote these $0$-dimensional strata; then we have
\[r \leq (n + 2) \cdot \mm(v),\]
because each trajectory has at most $n$ intermediate points of $\del X$, and $n + 2$ points in total.  Applying Part~\ref{models} of Lemma~\ref{stratification}, around each point $x_i$ we choose a neighborhood $U_i \subseteq X$ matching one of the local models, small enough that the various $U_i$ are disjoint, and let $D(U_i) \subseteq D(X)$ denote the double of $U_i$.  We take $U = \bigcup_{i = 1}^r D(U_i)$.  If $\mathcal{S}$ denotes the stratification on $D(X)$, then there exists some constant $M_n$ depending only on $n$, satisfying
\[\Vert [D(U_i), \del D(U_i)]\Vert_{\Delta}^\mathcal{S} \leq M_n\]
for all $i$.  Thus, the conclusion of the localization lemma gives
\[\Vert \alpha_*[D(X)]\Vert_{\Delta} \leq \Vert [U, \del U]\Vert_{\Delta}^\mathcal{S} \leq \sum_{i = 1}^r M_n \leq M_n \cdot (n+2)\cdot\mm(v).\]
\end{proof}

\begin{proof}[Proof of Theorem~\ref{special-case}]
We construct a degree-1 map $\alpha: D(X) \rightarrow M$ that sends all of the boundary $\del X$ to a single point.  Given such a map, we have $\alpha_*\pi_1(\del X) = 0$ (and $0$ is an amenable group), so Theorem~\ref{main-theorem} gives
\[\mm(v) \geq \const(n) \Vert \alpha_*[D(X)]\Vert_{\Delta} = \const(n) \Vert [M]\Vert_{\Delta} \geq \const(n) \cdot \Vol M,\]
where the value of $\const(n)$ is not fixed and may change between inequalities, but is always positive.  To construct $\alpha$, let $B$ be an open ball containing $\overline{U}$, and let $B'$ be a slightly smaller ball with $\overline{U} \subset B' \subset \overline{B'} \subset B$.  There is a degree-1 map $M \rightarrow M$ obtained by collapsing $\overline{B'}$ to a single point $*$, and stretching the cylinder $B \setminus B'$ to fill $B \setminus *$.  We define $\alpha$ on $X$ as the restriction of this map on $M$, and define $\alpha$ on the second copy of $X$ as the constant map at $*$.  Then $\alpha$ on all of $D(X)$ has degree $1$, and we have $\alpha(\del X) = *$.
\end{proof}

\section{Future directions}

One immediate follow-up question is how large the constant should be in Theorem~\ref{main-theorem}.  The $3$-dimensional case of Theorem~\ref{special-case} (Theorem~7.5 of~\cite{Katz09}) suggests that we might hope for a constant of $1$ for every $n$.  However, the constant obtained in our proof is much weaker and is a little confusing to compute.  It would be nice to compute an explicit upper bound for the stratified simplicial volume of the coordinate neighborhood of each trajectory of a versal vector field.  Also useful would be to check whether any examples might refute a possible constant of $1$ in Theorem~\ref{main-theorem}.

A second question for further study comes from a special case of Theorem~\ref{special-case}.  Let $f : M \rightarrow \mathbb{R}$ be a Morse function, and let $X$ be the space obtained by deleting a small open ball around each critical point of $f$.  If the negative gradient field $v = -\nabla f$ is traversally generic, then there are finitely many maximum-multiplicity trajectories in $X$, and if $f$ satisfies Morse-Smale transversality, then there are finitely many $n$-times-broken trajectories in $M$.  We might hope that these two sets of trajectories correspond in a fixed ratio depending on $n$.  Thus, Theorem~\ref{special-case} suggests the following conjecture.

\begin{conjecture}
Let $M$ be a closed, oriented hyperbolic manifold of dimension \mbox{$n + 1 \geq 2$}, and let $f : M \rightarrow \mathbb{R}$ be a Morse-Smale function.  Then we have
\[\#(n\text{-times-broken trajectories of }-\nabla f) \geq \const(n) \cdot \Vol M.\]
\end{conjecture}

Inconveniently, the proof method outlined above appears to require a lot of technical analysis to verify the traversally generic property.

\emph{Acknowledgments.} The authors would like to thank Larry Guth (Hannah's advisor) for initiating the collaboration, actively supervising most of our meetings, and suggesting improvements to the exposition in the paper.

\bibliography{bib-simpvol}{}
\bibliographystyle{amsplain}
\end{document}